\documentclass[11pt]{amsart}
\usepackage{amsmath}
\usepackage{amssymb}
\usepackage{amsfonts}
\usepackage{mathrsfs}
\usepackage{graphicx}
\usepackage{epsfig}
\usepackage[T1]{fontenc}

\usepackage{mathtools}

\newtheorem{lem}{Lemma}[section]
\newtheorem{theo}{Theorem}

\newcommand{\C}{{\bf C}}
\newcommand{\D}{{\bf D}}

\newcommand{\R}{{\bf R}}
\newcommand{\Z}{{\bf Z}}

\newcommand{\pr}{\mathrm{Pr}}

\newenvironment{block}[1]{\trivlist\item[\hskip \labelsep{{#1}.}]}{\endtrivlist}

\newtheorem{com}{Comment}
\newtheorem{exa}{Example}

\DeclareMathAlphabet{\mathpzc}{OT1}{pzc}{m}{it} % Zapf Chancery math alphabet

\newcommand{\cT}{\bf{T}}
\newcommand{\cR}{\bf{R}}
\newcommand{\cZ}{\bf{Z}}
\newcommand{\cS}{\bf{S}}

{\end{enumerate}\end{sc}}

\begin{document}
\title{Rays to renormalizations}
\date{\today}

\author{Genadi Levin}

\address{Institute of Mathematics, The Hebrew University of Jerusalem, Givat Ram,
Jerusalem, 91904, Israel}

%\author{Genadi Levin
%\\
%\small{Inst. of Math., Hebrew University, Jerusalem 91904, Israel}\\
%}
\normalsize
\maketitle

\begin{abstract}
Let $K_P$ be the filled Julia set of a polynomial $P$ and $K_f$ be the filled Julia set of a renormalization $f$ of $P$.
%$K_P$ the filled Julia set of $P$, $f$ a renormalization of $P$ and $K_f$ the filled Julia set of $f$.
We show, loosely speaking, that there is a finite-to-one function $\lambda$ from the set of $P$-external rays having limit points in $K_f$ onto the set of $f$-external rays to $K_f$ such that
$R$ and $\lambda(R)$ share the same limit set.
In particular, if a point of the Julia set $J_f=\partial K_f$ of a renormalization is accessible from $\C\setminus K_f$ then it is accessible through an external ray of $P$ (the inverse is obvious). Another interesting corollary is that: a component of $K_P\setminus K_f$ can meet $K_f$ only at a single (pre-)periodic point.
%where $K_P$ is the filled Julia set of $P$ meet $K_f$ at (pre-)periodic points.
We study also a correspondence
%$\Pi$
induced by $\lambda$ on arguments of rays.
%that assigns to the argument of a $P$-external ray accumulating in $K_f$ an argument of $f$-ray $\lambda(R)$,
%and show that $\Pi$ satisfies some properties that have been described in \cite{pz} in the case when $K_P$ is disconnected and $K_f$ is a component of $K_P$.
These results are generalizations to all polynomials (covering notably the case of connected Julia set $K_P$) of some results of \cite{LP}, \cite{abcdef} (Sect 6) and \cite{pz} where the case is considered when $K_P$ is disconnected and $K_f$ is a periodic component of $K_P$.
\end{abstract}

\section{Introduction}
\subsection{Polynomial external rays.}
Let $Q:\C\to\C$ be a non-linear polynomial considered as a dynamical system. Conjugating $Q$ if necessary by a linear transformation, one can assume without loss of generality that $Q$ is monic centered, i.e.,
$Q(z)=z^{\deg(Q)}+ a z^{\deg(Q)-2}+\cdots$.
We recall briefly necessary definitions, see e.g. \cite{dhorsay}, \cite{CG}, \cite{Mil0}, \cite{LS90-91} for details. The filled Julia set $K_Q$ of $Q$ is the complement $\C\setminus A_Q$ to the basin at infinity $A_Q=\{z: Q^n(z)\to\infty, n\to\infty\}$, and $J_Q=\partial A_Q=\partial K_Q$ is the Julia set (here and below $Q^n(z)$ is the image of $z$ by the $n$-iterate $Q^n$ of $Q$ for $n$ non-negative and the full preimage of $z$ by $Q^{|n|}$ for $n$ negative). Let $u_Q: A_Q\to\R_+$ be Green's function in $A_Q$ such that $u_Q(z)\sim \log|z|+o(1)$ as $z\to\infty$.
%for some (uniquely defined by $Q$) positive $C$.
%An equipotential of $Q$ of level $b>0$ is the level set $\{z: U_Q(z)=b\}$.
For all $z$ in some neighborhood $W$ of $\infty$,
$u_Q(z)=\log|B_Q(z)|$ where $B_Q$ is the B\"{o}ttcher coordinate of $Q$ at $\infty$, i.e., a univalent function from $W$
onto $\{w: |w|>R\}$, for some $R>1$, such that $B_Q(Q(z))=B_Q(z)^{\deg Q}$ for $z\in W$ and $B_Q(z)/z\to 1$ as $z\to\infty$.
%Hence, $u_Q(Q(z))=(\deg Q)  u_Q(z)$.
An equipotential of $Q$ of level $b>0$ is the level set $\{z: u_Q(z)=b\}$.
Alternatively, the equipotential containing a point $z\in A_Q$ is the closure of the union $\cup_{n>0}Q^{-n}(Q^n(z))$ and $u_Q(z)=\lim_{n\to\infty}(\deg(Q))^{-n}\log|Q^n(z)|$ is the level of this equipotential where $b=u_Q(z)$ is called the $Q$-level of $z\in A_Q$. Note that $u_Q(Q(z))=(\deg Q)  u_Q(z)$ for all $z\in A_Q$.
The gradient flow for Green's function (potentials) $u_Q$ equipped with direction from $\infty$ to $J_Q$ defines $Q$-external rays. More specifically, the gradient flow has singularities precisely at the critical points of $u_Q$ which are preimages by $Q^n$, $n=0,1,\cdots $ of critical points of $Q$ that lie in the basin of infinity $A_Q$.
If a trajectory $R$ of the flow that starts at $\infty$ does not meet a critical point of $u_Q$, it extends as a smooth (analytic) curve, external ray $R$, up to $J_Q$.
If $R$ does meet a critical point of $u_Q$, one should consider instead two corresponding (non-smooth) left and right external rays as left and right limits to $R$ of smooth external rays
(for a visualization of such rays, see e.g., Figures 1(a-b) of \cite{LP} or images in \cite{pz}-\cite{pz1}; to get an impression about geometry of the Julia set of renormalizable polynomials, see e.g. computer images of \cite{Pict}).
Each external ray $R$ is parameterized by the level of equipotential $b\in (+\infty,0)$.
%so that $R(b)$ is the point of intersection of $R$ with the equipotential of level $b$.
The argument $\tau\in\cT:=\cR/\cZ$ of an external ray $R$
is an argument of the curve $R$ asymptotically at $\infty$. Informally, $\tau$ is the argument at which $R$ crosses the "circle at infinity".
The correspondence between external rays and their arguments is one-to-one on smooth rays and two-to-one on non-smooth ones. If $R$ is a Q-external ray of argument $\tau$ then
$Q(R)$ is also a ray of argument $\sigma_{\deg(Q)}(\tau)$ where $\sigma_k(t)=tk (mod 1)$.
%, or just the argument of the standard ray $B_Q(R)$ where it is defined.
%Each external ray $R$ is parametrized by the level of equipotential $b\in (+\infty,0)$ so that $R(b)$ is the point of intersection of $R$ with the equipotential of level $b$.
%denoted then by $R^P_\tau$
Note that, for any $b$ large enough, $B_Q$ maps equiponential of level $b$ onto the round circle $\{|w|=e^b\}$
and arcs of external rays from this equipotential to $\infty$ onto standard rays that are orthogonal to this circle. Finally, $K_Q$ is connected if and only if $B_Q$ extends as a univalent function to the basin of infinity $A_Q$ and if and only if all external rays of $Q$ are smooth.
\iffalse
The filled Julia set $K_P$ of $P$ is the complement $\C\setminus A_P$ to the basin at infinity $A_P=\{z: P^n(z)\to\infty, n\to\infty\}$, and $J_P=\partial A_P=\partial K_P$ is the Julia set. Let $U_P: A_P\to\R_+$ be the Green function in $A_P$ such that $U_P(z)\sim C\log|z|$ as $z\to\infty$ for some (uniquely defined by $P$) positive $C$. An equipotential of $P$ of level $b>0$ is the level set $\{z: U_P(z)=b\}$. The gradient flow to equipotentials (equipped with direction from $\infty$ to $J_P$) defines $P$-external rays. If an external ray $R$ does not meet critical points of $U_P$, it is a smooth (analytic) curve. If $R$ meets a critical point of $U_P$, one should consider instead two corresponding (non-smooth) left and right external rays as left and right limits to $R$ of smooth external rays.
Each external ray $R$ is parameterized by the level of equipotential $b\in (+\infty,0)$ so that $R(b)$ is the point of intersection of $R$ with the equipotential of level $b$.
The argument $\tau\in\cT:=\cR/\cZ$ of a $P$-external ray $R$
is an argument of the curve $R$ asymptotically at $\infty$. Informally, $\tau$ is the argument at which $R$ crosses the "circle at infinity".
%Each external ray $R$ is parametrized by the level of equipotential $b\in (+\infty,0)$ so that $R(b)$ is the point of intersection of $R$ with the equipotential of level $b$.
%denoted then by $R^P_\tau$
Note that $K_P$ is connected if and only if all external rays of $P$ are smooth.
\fi

Let ${\cS}=\{|z|=1\}$ be the unit circle which we identify - when this is not confusing - with $\cT$ via the exponential $t\in{\cT}\mapsto\exp(2\pi i t)\in\cS$.
\subsection{Polynomial-like maps and renormalization.}
Let us recall, \cite{dh}, that
%the notion of a polynomial-like map \cite{dh}.
%and introduce an equivalence relation on such maps with connected Julia sets.
a triple $(W, W_1, f)$ is a polynomial-like map,
if $W,W_1$ are topological disks, $\overline{W_1}\subset W$ and $f: W_1\to W$ is a proper holomorphic map
of some degree $m\ge 2$.
The set of non-escaping points $K_f=\cap_{n=1}^\infty f^{-n}(W)$ is called the filled Julia set of $(W, W_1, f)$.
%Assume $K_f$ is connected.
By the Straightening Theorem \cite{dh}, there exists a
%unique
monic centered polynomial $G$ of degree $m$ which is hybrid equivalent to $f$, i.e.,
%(ii) topological disks $V_1, V$ with smooth boundaries such that
%$K_f\subset{V_1}\subset\overline{V_1}\subset W_1\subset\overline{W_1}\subset V\subset\overline {V}\subset W$, and $f:V_1\to V$ is a proper map of degree $m$,
there is a quasiconformal
homeomorphism $h:\C\to\C$
%of a neighborhood of $K_f$ onto a neighborhood of $K_G$
which is conformal a.e. on $K_f$, such that
$G\circ h=h\circ f$ near $K_f$. The map $h$ is called \emph{straightening}.
This implies in particular that $K_f$ is the set of limit points of the union $\cup_{n\ge 0}f^{-n}(z)$, for any $z\in W$ with, perhaps, at most one exception.

We say that another polynomial-like map $(\tilde W, \tilde W_1, \tilde f)$
of the {\it same degree} $m$ is {\it equivalent} to $(W, W_1, f)$ if there is a component $E$ of $W\cap\tilde W$ such that $K_f\subset E$ and $f=\tilde f$ in a neighborhood of $K_f$.
Taking a point $z$ as above close to $J_f=\partial K_f$, it follows (cf. \cite[Theorem 5.11]{mcm}) that $K_f=K_{\tilde f}$ and that this is indeed an equivalence relation on polynomial-like maps. Denote by ${\bf f}$ the equivalence class of the polynomial-like map $(W, W_1, f)$,
$K_{\bf f}$, $J_{\bf f}$ the corresponding filled Julia set and Julia set of (any representative of) ${\bf f}$,
and by $f$ a restriction to a neighborhood of $K_{\bf f}$ of the map of a ${\bf f}$-representative (i.e., for any two representatives $(W^{(i)}, W^{(i)}_1, f_i)$, $i=1,2$, we have $f_1=f_2=f$ in a neighborhood of $K_{\bf f}$).

From now on, let us fix a monic centered polynomial $P:\C\to\C$ of degree $d>1$.

We say that ${\bf f}$ is a {\it renormalization} of $P$ (cf. \cite{mcm}, \cite{Inou}) if ${\bf f}$ is an equivalence class of polynomial-like maps
such that $K_{\bf f}$ is a connected proper subset of $K_P$
and, for some $r\ge 1$, $f=P^r$ in a neighborhood of $K_{\bf f}$.
\subsection{Assumptions.} Suppose that

%We assume:
\

(p1) {\it $\bf f$ is a renormalization of $P$
%is a renormalization of $P$
%, i.e., ${\bf f}$ is a class of equivalence of polynomial-like maps
%such that $K_{\bf f}$ is a connected proper subset of $K_P$
%and, for some $r\ge 1$, $f=P^r$ in a neighborhood of $K_{\bf f}$
}.

\

To avoid a situation when an external ray of $P$ can have a limit point in $J_{\bf f}$ as well as a limit point off $J_{\bf f}$, we introduce another condition:

\

(p2) {\it there exists a representative $(W^*, W^*_1, f)$ of the renormalization ${\bf f}$ of $P$ and some $b_*>0$ as follows. If $z\in\partial W^*_1$ belongs to an external ray of $P$ which has a limit point in $K_{\bf f}$ then the $P$-level of $z$ is at least $b_*$, i.e., $u_P(z)\ge b_*$.}

\

Let us stress that external rays of $P$ as in (p2) can cross boundaries of $W^*$, $W^*_1$ many times (or e.g. have joint arcs with the boundaries).

This condition holds if $W^*$ is obtained by the following frequently used construction that we only indicate here, see \cite{Mil1}, \cite{mcm}, \cite{Inou} for details.
In the first step, a simply-connected domain $W_0$ is built using an appropriate Yoccoz puzzle so that $\partial W_0=L_{hor}\cup L_{vert}\cup F$ where $L_{hor}$ is the union of a finitely many arcs of a fixed equipotential of $P$, $L_{vert}$ is the union
of a finitely many arcs of external rays of $P$ between ends of arcs of $L_{hor}$
and a finite set $F$ of some repelling periodic points of $J_P$ or/and their preimages such that
$K_{\bf f}\subset W_0\cup F$ and $f: f^{-1}(W_0)\to W_0$ is an branched covering.
By the construction,
every external ray of $P$ to $J_{\bf f}\setminus F$ must cross  the "horizontal" part $L_{hor}$ so that (p2) is obviously satisfied for the set of those rays.
If either $L_{vert}=F=\emptyset$ (as in Example \ref{example} that follows) or $F\cap K_{\bf f}=\emptyset$, one can take $W^*=W_0$ so that (p2) holds for $W^*_1=f^{-1}(W^*)$.
If $F\subset J_{\bf f}$, then $W_0\setminus f^{-1}(W_0)$ is a degenerate annulus.
%By the construction,
%every external ray of $P$ to $J_{\bf f}\setminus F$ must cross  the "horizontal" part $L_{hor}$ so that (p2) is obviously satisfied for the set of those rays.
Then, in the second step, $W^*$ is modified from $W_0$ by "thickening" \cite[p.12]{Mil1} around points of the set $F$ which adds only finitely many rays (tending to $F$). Then (p2) holds for $W^*_1=f^{-1}(W^*)$ as well.

%consists of a
%finitely many arcs of equipotentials of $P$,
%arcs of smooth external rays of $P$ that either join those arcs of equipotentials or begin at points of
%those equipotentials and tend to $J_P$,
%and subcontinua of $K_P$\footnote{Such $W$ exists in all (known to the author) examples of renormalizations of $P$.}

%Condition (p2) is an extra assumption\footnote{The author is not aware of an example when (p2) breaks down.}.

%The condition (p2) can be satisfied in all (known to the author) examples by modifying in the condition (p1) $W$ if necessary using an appropriate Yoccoz's puzzle.
%, see~\cite{mcm} for the general theory.
%\footnote{
%The condition (p2) would be violated if, for example, the complement $K_P\setminus K_f$ contained a sequence of components forming a "comb"
%On the other hand, the condition (p2) excludes, for example, the following picture: the complement $K_P\setminus K_f$ contains a sequence of components forming a "comb"}.
%However, in interesting cases (pl1) implies (pl2) after perhaps modifying $W$ using the Yoccoz puzzle construction, e.g. \cite{inou} ??????????????????????????????????????
\begin{exa}\label{example} Assume that the Julia set of the polynomial $P$ is disconnected and $K$ is a component of $K_P$ different from a point. In this case $K=K_{\bf f}$ for some renormalization ${\bf f}$ of $P$ and conditions (p1)-(p2) are fulfilled.
% if $K_P$ is disconnected and $K_f$ is a component of $K_P$ different from a point.
%, the case which is covered by Theorem 6.9 of \cite{abcdef} and by \cite{pz}.
The boundary of $W^*$ (hence, $W^*_1$, too) can be chosen to be merely a component of an equipotential that encloses $K$.
%???????????This case has been done in \cite{LP}, \cite{abcdef}, and \cite{pz}.
With such a choice, each intersection point of an external ray of $P$ with $\partial W^*$ has a fixed  level so every external ray can cross boundaries $W^*$ and $W^*_1$  at most once.
\end{exa}

%Let $m$ be the degree of (any representative of) ${\bf f}$.
%By \cite{dh}, since $K_{\bf f}$ is connected, the monic centered polynomial $G$ of degree $m$ to which a representative of ${\bf f}$ is hybrid equivalent is unique defined, i.e., independent on the representative and the straigthening.

%Let the following condition hold:

%(p1) {\it Suppose that $(W, W_1, f)$ is a polynomial-like mapping such that, for some $r\ge 1$, $f=P^r$ in a neighborhood of $K_f$, moreover, $K_f$
%is a connected proper subset of $K_P$.}

%The class of equivalence of $(W, W_1, f)$ will be called
%a {\it renormalization} ${\bf f}$ of $P$, cf. ???\cite{mcm}.
%\begin{equation}\label{pl}
%f=P^r|_{W_1}: W_1\to W
%\end{equation}
%is a polynomial-like map~\cite{dh}, i.e., $W,W_1$ are topological disks, $\overline{W_1}\subset W$ and $f: W_1\to W$ is a proper map
%of a degree $m\ge 2$. We assume that
%We make the following two assumptions:
%\begin{enumerate}
%(p1) {\it the filled Julia set $K_f=\cap_{n=1}^\infty f^{-n}(W)$
%of $f$ is connected, and it is a proper subset of $K_P$.}
%Let us define polynomial-like rays.
%Given a renormalization ${\bf f}$ of $P$,
%\subsection{Polynomial-like rays.}
Our goal is to study a correspondence between external rays of $P$ that have limit points in $J_{\bf f}$, on the one hand, and external, or polynomial-like rays of the renormalization ${\bf f}$, on the other
(up to change of straightening, see below). In the case of disconnected Julia set $J_P$ and the renormalization ${\bf f}$ as in Example \ref{example}
this has been done in \cite{LP}, \cite{abcdef} (Sect. 6), and \cite{pz}.
\subsection{Polynomial-like rays.}
%Recall that, for a curve $\alpha:(0,1)\to\C$ so that there exists the beginning point $\lim_{t\to 0}\alpha(t)$, either finite or infinite, a point $z\in\C$ is a limit (or accumulation) point of $\alpha$ if $z=\lim_{n\to\infty}\alpha(t_n)$
%for some sequence $t_n\to 1$. The limit (or principal, or accumulation) set of $\alpha$ is the set $\pr(\alpha)$ of limit points of $\alpha$. Later on, $\pr(\alpha)$ is always a compact, hence, connected, subset of $\C$.
For a curve $\alpha:[0,1)\to\bar{\C}$,
%so that there exists the beginning point $\lim_{t\to 0}\alpha(t)$, either finite or infinite, $z\in\C$ is a limit (or accumulation) point of $\alpha$ if $z=\lim_{n\to\infty}\alpha(t_n)$
%for some sequence $t_n\to 1$.
the limit (or principal, or accumulation) set of $\alpha$ is the set $\pr(\alpha)=\overline\alpha\setminus\alpha$.
%of limit points of $\alpha$. Later on, $\pr(\alpha)$ is always a compact, hence, connected, subset of $\C$.

Let us define external rays of the renormalization ${\bf f}$.
By \cite{dh}, since $K_{\bf f}$ is connected, the monic centered polynomial $G$ of degree $m$ which is hybrid equivalent to any representative of ${\bf f}$ is unique defined by ${\bf f}$.
%, i.e., independent on the representative and the straigthening.
Let $h$ be a {\it straightening} of ${\bf f}$. By this we mean a quasiconformal homeomorphism $\C\to\C$
%in a neighborhood of $K_{\bf f}$
which is conformal
a.e. on $K_{\bf f}$ and $G\circ h=h\circ f$ holds on some
%$V_1=f^{-1}(V)$ where $V$ is a
neighborhood of $K_{\bf f}$.
%with a smooth boundary
%such that $V_1$ is compactly contained in $V$.
One can also assume that $h$ is conformal at $\infty$ such that $h'(\infty)\neq 0$.
%conjugates $f$ to $G$ in some neighborhood of $K_{\bf f}$.
As the filled Julia set $K_G$ of the polynomial $G$ is connected, given $t\in\cT$ there is one and only one external ray of $G$
of argument $t$, denoted by $R_{t, G}$.
The $h^{-1}$-image
$l_t^h:=h^{-1}(R_{t,G})$ of the ray is called the {\it polynomial-like ray to
$K_{\bf f}$ of argument} $t$.
As $h:\C\to\C$ is a homeomorphism,
$\pr(l_t^h)=h^{-1}(\pr(R_{t,G}))$.
Note that the straightening $h$ is not unique. However, the polynomial $G$ is unique,
and if $\tilde h$ is another straightening,
%$\tilde h\circ h^{-1}$ extends
%to a quasiconformal homeomorphism of $\C$ which is conformal a.e. on $K_G$ and commutes with $G$ near $K_G$.
although $\tilde h$ defines another systems of polynomial-like rays,
%they are equivalent in a sense that
the homeomorphism $\tilde h^{-1}\circ h: \C\to\C$ maps $l_t^h$ onto $l_t^{\tilde h}$ and $\pr(l_t^h)$ onto $\pr(l_t^{\tilde h})$).

In what follows we fix a straightening map
$h:\C\to\C$ (see Theorem \ref{t2}(e) and its proof though).
%(for a later application, one can assume that $h$ is conformal in a neighborhood of $\infty$ and $h'(\infty)=1$).
%where $V$ is a small enough neighborhood of $K_{\bf f}$ with a smooth boundary such that $V_1=f^{-1}(V)$ is compactly contained in $V$.
%and a representative $(W, W_1, f)$ of the renormalization ${\bf f}$ of $P$ that satisfies (p2).
Then the set of polynomial-like rays $\{l_t\}$ is fixed, too (where we omit $h$ in $l_t^h$ as $h$ is fixed).
For brevity, $P$-external rays are called $P$-{\it rays}, or just {\it rays}, and polynomial-like rays to $K_{\bf f}$ are $f$-{\it rays, or polynomial-like rays}.
\subsection{Main results.}
%For a curve $\alpha:[0,1)\to\bar{\C}$,
%so that there exists the beginning point $\lim_{t\to 0}\alpha(t)$, either finite or infinite, $z\in\C$ is a limit (or accumulation) point of $\alpha$ if $z=\lim_{n\to\infty}\alpha(t_n)$
%for some sequence $t_n\to 1$.
%the limit (or principal, or accumulation) set of $\alpha$ is the set $\pr(\alpha)=\overline\alpha\setminus\alpha$.
%of limit points of $\alpha$. Later on, $\pr(\alpha)$ is always a compact, hence, connected, subset of $\C$.
Given a connected compact set $K\subset\C$ which is different from a point, we say
%(abusing probably the standard terminology)
that a curve $\gamma:[0,1)\to\Omega:=\C\setminus K$
%having its limit set on $\partial K$
{\it converges
to a prime end} $\hat P$ of $K$ if, for a conformal homeomorphism $\psi:\C\setminus K\to \{|z|>1\}$, the curve $\psi\circ\gamma:[0,1)\to \{|z|>1\}$ converges to a single
point $P\in\cS$; we say that $\gamma$ converges to the prime end $\hat P$ {\it non-tangentially} if, moreover, $\psi\circ\gamma$ converges to the point $P$ non-tangentially, i.e., the set
$\psi\circ\gamma((1-\epsilon,1))$ lies inside a sector (Stolz angle)
$\{z: |\arg(z-P)-\arg P|\le\alpha\}$, for some $\epsilon>0$, $\alpha\in(0,\pi/2)$. Furthermore, we say
that
two curves $\gamma_1,\gamma_2:[0,1)\to\Omega$ are {\it K-equivalent if they both converge to the same prime end and, moreover, have the same limit sets} $\pr(\gamma_1)=\pr(\gamma_2)$ in $\partial K$\footnote{One can show that, provided $\gamma_1$ converges to a single point $a\in\partial K$, $\gamma_2$ is K-equivalent to $\gamma_1$ if and only if $\gamma_1$, $\gamma_2$ are homotopic through a family of curves in $\Omega$ converging to $a$.}.
By Lindel\"{o}f's theorem, e.g., \cite{Pom}, Theorem 2.16, if two curves converge to the same prime end non-tangentially, they share the same limit set. Therefore,
if $\gamma_1,\gamma_2$ converge to the same prime end of $K$ non-tangentially, then $\gamma_1,\gamma_2$ are also K-equivalent.
%One can also prove that provided $\pr{\gamma_1}$ is a single point $a\in\partial K$, $\gamma_2$ is K-equivalent to $\gamma_1$ if and only if $\gamma_1$, $\gamma_2$ are homotopic through a family of curves in $\Omega$ converging to $a$.
%two curves $l_1, l_2:[0,1)\to\C\setminus K$ such that $\pr(l_1)=\pr(l_2)\subset\partial K$ are homotopic
%among curves with the same limit set
%(i.e., rel. this common limit set),
%if $l_1,l_2$ correspond to the same prime end of $K$, i.e., for a conformal homeomorphism $\psi:\C\setminus K\to \{|z|>1\}$, the curves $\psi(l_1)$ and $\psi(l_2)$ in $\C\setminus \{|z|>1\}$ both converge
%to the same point of the unit circle: $\pr(\psi(l_1))=\pr(\psi(l_2))$ and is a single point of $\cS$.

The following statement was proved in \cite{abcdef}\footnote{In \cite{abcdef}, a different terminology is used.} in the set up of Example \ref{example}.
\begin{theo}\label{T1} (cf. \cite{abcdef}, Theorem 6.9)  Assume (p1)-(p2) hold.
For each $P$-ray $R$ that has an accumulation point in $K_{\bf f}$ we have that $\pr(R)\subset J_{\bf f}$ and
there is one and only one polynomial-like ray $l=\lambda(R)$ such that two curves $l$, $R$ are $K_{\bf f}$-equivalent. Moreover,
$l$, $R$ converge to a single prime end of $K_{\bf f}$ non-tangentially.
%$\pr(l)=\pr(R)$ and the curves $l$ and $R$ are homotopic in $\mathbb{C}\setminus K_{\bf f}$ among curves
%with the same limit set.
%\footnote{??? We mean by this that two curves $\psi(R)$ and $\psi(\lambda(R))$ converge non-tangentially
%to the same point on the unit circle $\cS$ where $\psi: \C\setminus K_f\to \{|z|>1\}$ is a conformal homeomorphism???.}.
%To every polynomial-like ray $l$ to $E$
%there corresponds at least one external ray $R$ of $P$,
%so that $\Pi(l)=\Pi(R)$, and the curves $l$, $R$ are homotopic
%in $\mathbb{C}\setminus E$ among all curves with the same limit set.
Furthermore, $\lambda: R\mapsto l$ maps the set of P-rays to $K_{\bf f}$
\textbf{onto} the set of polynomial-like rays, and is ``almost injective'':
$\lambda$ is one-to-one except when one and only one of the following (i)-(ii)
holds. Suppose that $\lambda^{-1}(\ell)=\{R_1,\dots,R_k\}$ with $k>1$.

\begin{enumerate}

\item[(i)] $k=2$ and both rays $R_1, R_2$ are non-smooth and share a common arc starting at a critical point of Green's function $u_P$
to $J_{\bf f}$, or

%(ii) some iterate $f^n(\Pi(R))$ of $\Pi(R)$ contains a critical point of $P$,

\item[(ii)] there is $z\in J_{\bf f}$ such that $\pr(R_i)=\{z\},
    i=1,\dots, k$, at least two of the rays $R_1, \dots, R_k$ are disjoint,
    and, for some $n\ge 0$, $P^{rn}(z)\in Y$ where $Y\subset J_{\bf f}$ is a finite collection
of repelling or parabolic periodic points of $P$ that depends merely on $K_{\bf f}$.

\end{enumerate}
If $K_P$ is connected then (i) is not possible.
%of periodic points of $P^p: E\to E$,
%each of which is either repelling or parabolic.
\end{theo}
Note that in case (ii) any two disjoint $P$-rays completed by the joint limit point $z$ split the plane into two domains such that one of them
contains $K_{\bf f}\setminus\{z\}$ and another one - points from $K_P\setminus K_{\bf f}$. In particular, if $K_P$ is connected,
the second domain must contain a component of $K_P\setminus K_{\bf f}$ that goes all the way to a pre-periodic point $z\in J_{\bf f}$.
In fact, this is "if and only if": see the following Theorem \ref{t1} (b).

As an illustration, see e.g. pictures on p.116 of \cite{mcm} (explained in Example IV, p. 115) of a "dragon" filled Julia set of a quadratic polynomial $P$ admitting three
renormalizations; corresponding to theses renormalizations maps $\lambda$ are one-to-one except at a countably many polynomial-like rays where $\lambda$ is 6-to-1 at the top picture,
2-to-1 at the left bottom and 3-to-1 at the right bottom. In all three cases, landing points of rays where $\lambda$ is not one-to-one are (pre-)periodic to a fixed point of $P$
where six $P$-rays land.
%(at the top, the black component in the middle is the filled Julia set $K_{\bf f}$ of a renormalization, and each component of $K_P\setminus K_{\bf f}$ touches $K_{\bf f}$
%at a preimage of a fixed point of $P$).

%In fact, this is "if and only if": see the following Theorem \ref{t1} (b).

Next two theorems are consequences of the proof of Theorem \ref{T1}.
\begin{theo}\label{t1} Assume (p1)-(p2).

(a) If a point $a\in J_{\bf f}$ is accessible along
a curve $s$ in $\C\setminus K_{\bf f}$, then $a$ is the landing point
of a $P$-ray $R$, moreover, curves $s$, $R$ are $K_{\bf f}$-equivalent.
%\
%??? Was this proved:
%which is homotopic to $s$ among curves in $\C\setminus K_{\bf f}$ converging to $a$
%????

%\

(b) There exists a finite set $Y\subset J_{\bf f}$ of repelling
or parabolic periodic points of $f$, as follows. Let $S$ be a component of $K_P\setminus K_{\bf f}$ such that
$(\overline{S}\setminus S)\cap J_{\bf f}\neq\emptyset$. Then $\overline{S}\setminus S$ is a single point $b\in J_{\bf f}$ and, moreover,
$f^n(b)\in Y$ for some $n\ge 0$.
%Let $a\in J_f$. Assume that
%the point $a$ is accessible from $K_P\setminus K_f$, i.e., there is a
%curve $s$ converging to $a$, such that $s\subset K_P\setminus K_f$.
%Then, for some $n\ge 0$, $f^n(a)\in Y$.
%Every point of $Y$ satisfies (b1), in particular, at least two rays land at it.
\end{theo}
Note that part (a)
%generalizes \cite{LP}, Theorem 2 and
is in fact an easy corollary of Lemma \ref{l1}, Sect \ref{s2}, similar to \cite{LP}.
%Note also that the homotopy of the curves $s$ and $R$ (as in part (a)) implies that they are also $K_{\bf f}$-equivalent, see the proof of Theorem \ref{t1}.
Part (b) is void if (and only if) $K_{\bf f}$ is itself a component of $K_P$.

For the next statement, we introduce the following notations. Let $\Lambda\subset \cT$ be the set of
arguments of all $P$-rays that have their limit points in $J_{\bf f}$.
%By Theorem \ref{T1}, the set $\Lambda$ depends only on the renormalization ${\bf f}$ of $P$.
Observe that by Theorem \ref{T1} the whole limit sets of such rays are in $J_{\bf f}$ and, given $\tau\in\Lambda$, there is one and only one $P$-ray denoted then by $R_{\tau,P}$ which has its limit set in $J_{\bf f}$.
Indeed, this is obvious if the $P$-ray of argument $\tau$ is smooth. On the other hand, if there are two $P$-rays, left and right, of argument $\tau$, only one of them can have its limit point in $J_{\bf f}$
because the other one must go to another component of $K_P$.
%on the other hand, the unit circle $\cT$ is the set of arguments of all polynomial-like rays.
Now, the map $\lambda$ of Theorem~\ref{T1} induces a map $p: \Lambda\to \cT$, such that for all $\tau\in\Lambda$,
$$\lambda(R_{\tau,P})=l_{p(\tau)}.$$
By Theorem \ref{T1}, $\pr(l_{p(\tau)})=\pr(R_{\tau,P})$ and, moreover, $R_{\tau,P}$, $l_{p(\tau)}$ are $K_{\bf f}$-equivalent.
%and the curves $l_{p(\tau)}$, $R_{\tau,P}$ are homotopic in $\mathbb{C}\setminus K_{\bf f}$ rel. the same limit set.
%Next corollary easily follows from definitions and the proof of Theorem \ref{T1}, see Sect. \ref{s4}.
%In the case of Example \ref{example}, i.e., $K_P$ is disconnected and $K_f$ is a periodic component of $K_P$, it is proved in \cite{pz}
%(by a different method),
%with extra statements about Hausdorff dimension of $\Lambda$ and a bound for the cardinality of fibers of the map $p$.

Given a positive integer $k$, let $\sigma_k:\cT\to \cT$, $\sigma_k(t)=k t (mod 1)$. Recall that $\deg(f)=m$. Let $D:=\deg(P^r)=d^r$.
\begin{theo}\label{t2} (cf. \cite{pz})
%The map $\lambda$ of Theorem~\ref{T1} induces a map $p: \Lambda\to S^1$ where $\Lambda\subset S^1$ is the set of
%arguments of all $P$-rays with their limit sets in $J_f$ and the unit circle $S^1$ is the set of arguments of all polynomial-like rays, such that for all $\tau\in\Lambda$,
%$$\lambda(R^P_\tau)=l_{p(\tau)}.$$
%The map $p$ has the following properties:
\begin{enumerate}
\item[(a)] $\Lambda$ is a compact nowhere dense subset of $\cT$ which is invariant under $\sigma_{D}$,
\item[(b)] $\sigma_m\circ p=p\circ\sigma_{D}$ on $\Lambda$,
\item[(c)] the map $p: \Lambda\to \cT$ is surjective and finite-to-one, moreover, "almost injective" as defined in Theorem \ref{T1},
\item[(d)] $p: \Lambda\to \cT$ extends to a continuous monotone degree one map $\tilde p: \cT\to \cT$.
\item[(e)] the map $p$ is unique in the following sense: if $\tilde p:\Lambda\to\cT$ corresponds to another straightening $\tilde h$,
then $\tilde p(t)=p(t)+k/(m-1) (mod 1)$, for some $k=0,1,\cdots m-1$.
%up to a postcomposition by a map $T_k: t\mapsto t+k/(m-1) (mod 1)$, for some $k=0,1,\cdots m-1$: if $\tilde p:\Lambda\to\cT$ corresponds to another straightening $\tilde h$
\end{enumerate}
\end{theo}
In the set up of Example \ref{example}, i.e., $K_P$ is disconnected and $K_f$ is a periodic component of $K_P$, Theorem \ref{t2} was proved in \cite{pz}
(by a different method), with part (c) replaced by an explicit bound for the cardinality of fibers of the map $p$ as well as
with an extra statement about the Hausdorff dimension of the set $\Lambda$.
% and a bound for the cardinality of fibers of the map $p$.

\

Detailed proof of the main Theorem \ref{T1} is contained in Sect \ref{s2} and proofs of Theorems \ref{t1}-\ref{t2} are in Sect. \ref{s3}.
The proof of Theorem \ref{T1} follows rather closely proofs of Lemma 2.1 of \cite{LP} and Theorems 6.8-6.9 of \cite{abcdef}. An essential difference
is that we have to adapt the proofs to the situation that external rays of $P$ can cross the boundary of $W^*_1$ as in (p2) many times.
%\begin{com}\label{b1}
%In fact, we prove also the following.
%Assume that the polynomial-like map
%~(\ref{pl})-(pl2).
%Suppose a point $a\in J_f$ is the landing point of two rays $R', R''$,
%such that one of the two domains with the boundary curve
%$R'\cup R''\cup \{a\}$
%contains a point of $K_P$, and is disjoint with $K_f$. Then
%the conclusion of part (b) of Corollary~\ref{t1} holds.
%See the proof of Corollary~\ref{t1}, part (b).
%
%\end{com}
%\begin{com}\label{kiwi}
%Every repelling or parabolic periodic point $a$ of $P$ which
%is not a single point component of $J_P$ (i.e., every point of $Y$),
%is the landing point of at least one
%and at most finitely many rays, and every ray landing at $a$
%has the same period. For connected Julia set, this is proved by
%Douady, see~\cite{Hu}, and for general case, see~\cite{EL} and~\cite{LP}.
%Rotation number
%of $a$ is well-defined and rational. It is zero if and only
%if $P^q$ fixes every ray landing at $a$, where $q$ is the period of $a$
%by $P$.
%Moreover, if $P$ has no irrational neutral periodic points, then $K_P$ is
%(openly) locally connected at every repelling or
%parabolic periodic point of $P$~\cite{kiwi}.
%\end{com}

{\bf Acknowledgments.}
%We thank Feliks Przytycki for helpful discussions.
%Writing up this note was motivated by
%a question asked by Sasha Blokh and partially answered by the author in December, 2012.
%By its content and proof, the present note is a follow-up to~\cite{LP} and Sect. 6 of~\cite{abcdef}.
In \cite{L2012}, we answered, under an extra assumption, a question of Alexander Blokh to the author whether an accessible point of the filled Julia set $K_f$ of a renormalization
$f$ of $P$ by some curve outside of $K_f$ is always accessible by an external ray of $P$ (i.e., by a curve outside of the filled Julia set $K_P$ of $P$).
Theorem \ref{t1} (a) straightens this answer, under a weaker assumption (p2).
%Ideas, constructions and proofs are slightly adapted the ones of \cite{LP}, \cite{abcdef}.
Theorem \ref{t2} was added following a recent work \cite{pz} which also served as an inspiration for writing this paper up.
Finally, we would like to thank Feliks Przytycki for a helpful discussion 
%leading particularly to Comment \ref{another},
and the referee for comments that helped to improve the exposition.
%Our proofs are minor adaptations of \cite{LP} and \cite{abcdef}, Sect. 6.
\section{Proof of Theorem~\ref{T1}}\label{s2}
Let $f:W^*_1\to W^*$ be a representative of ${\bf f}$ as in (p2).
As $K_f$ is connected, all the critical points of the map
$f: W^*_1\to W^*$ are contained in $K_f$. Hence, for each $k$, $f^k:f^{-k}(W^*_1\setminus K_f)\to W^*_1\setminus K_f$
is an unbranched (degree $m^k$) map. Therefore, $L_k:=f^{-k}(\partial W^*_1)$ is the boundary of a simply-connected domain $f^{-k}(W^*_1)$.
Let $\mathcal{R}$ denote a set of all $P$-rays $R$ such that $R$ has a limit point in $J_f$.
%Note that $R\in\mathcal{R}$ implies $P^r(R)\in\mathcal{R}$.
First, we show that all limit points of $R\in\mathcal{R}$ are in $J_f$ introducing along the way some notations.
%Denote $L_k=f^{-k}(\partial W_1)$, $k=0,1,\cdots$.
Let $b_{*,k}=\inf\{u_P(z)| z\in R\cap L_k, R\in\mathcal{R}\}$.
By (p2), $b_{*,0}>0$. As $R\in\mathcal{R}$ implies $P^r(R)\in\mathcal{R}$, then
$b_{*,k}\ge b_{*,0}/D^k$, hence, $b_{*,k}>0$, for all $k$.
Let $R\in\mathcal{R}$ and $k\ge 0$. Since $R\cap L_k$ is a closed set and $b_{*,k}>0$,
there exists a unique point $z_k(R)\in R\cap L_k$ so that $u_P(z_k(R))=\inf\{u_P(z)| z\in R\cap L_k\}$.
Observe that the arc $\Gamma_{k,R}$ of $R$ from $z_k(R)$ down to $J_P$ belongs entirely to $\overline{f^{-k}(W^*_1)}$.
As $\cap_{k\ge 0}\overline{f^{-k}(W^*_1)}=K_f$, we get immediately that the limit set of $R$, which is $\cap_{k\ge 0}\overline{\Gamma_{k,R}}$, is a subset of $J_f$.

Before proceeding with more notations and the main lemma, let us note that $b_{*,k}=b_{*,0}/D^k$, $k=1,2,\cdots$.
Indeed,
as $f^k:f^{-k}(W_1\setminus K_f)\to W_1\setminus K_f$
is an unbranched covering,
each component of the set $f^{-k}(R)$ is an arc of some ray from $\mathcal{R}$.
This implies that $b_{*,k}\le b_{*,0}/D^k$. The opposite inequality was seen before.

%Proof of part (a) follows closely the proof of Theorem 2 of~\cite{LP},
%see also~\cite{abcdef}.
%Lemma~\ref{l1} is a minor modification of Lemma 2.1~\cite{LP}.
%In turn, part (b) will be a consequence of Lemma~\ref{l1} and
%Theorem 6.8 of~\cite{abcdef}.
Now, choose a conformal isomorphism
$\psi$ from ${\bf C}\setminus K_f$ onto
${\bf D}^*=\{|z|>1\}$, such that $\psi(z)/z\to e$ as $z\to \infty$,
for some $e>0$.
A curve $\tilde{R}$ in ${\bf D}^*$ with a limit set in ${\cS}=\{|z|=1\}$
is called a $K$-related ray, if its preimage $\psi^{-1}(\tilde{R})$
is a $P$-ray $R\in\mathcal{R}$, i.e., $R$ has its limit set in $K_f$.
The argument of $\tilde{R}$ is said to be the argument of the ray $\psi^{-1}(\tilde{R})$.
Let $A_K=\psi(W^*\setminus K_f)$ be an "annulus" with boundary curves $\psi(\partial W^*)$ and $\cS$.
Denote $\tilde z_k(\tilde R)=\psi(z_k(R)$. Note that $\tilde z_k(\tilde R)\in\psi(L_k)\cap\tilde R$ and the arc of the R-related ray $\tilde R$
from $\tilde z_k(\tilde R)$ to $\cS$ is contained in an "annulus" between $\psi(L_k)$ and $\cS$. An arc of a $K$-related
ray $\tilde R=\psi(R)$ from the point $\tilde z_0(\tilde R)=\psi(z_0(R))\in \psi(L_0)$ to $\cS$ is called a $K$-related arc.
Its argument is the argument of the corresponding ray.
The following main lemma and its proof are minor adaptations of the ones of~\cite{LP}, Lemma 2.1.
\begin{lem}\label{l1}

1$^o$ Every $K$-related arc has a finite length and, hence, converges to
a unique point of $\cS$.
%The lengths of the parts of the rays where
%the function $M\circ \Phi^{-1}\ge t$ converge
%to 0 uniformly (exponentially fast) as $t\to\infty$.

2$^o$ For every closed arc $I\subset \cS$
(particularly, a point), the set $K(I)$ of arguments
of all $K$-related arcs converging to a point
of $I$ is a non-empty compact set.

3$^o$ The set of all $K$-related arcs in $\{z:1<|z|<1+\epsilon \}$
converging to a point $z_0$ lies in a Stolz angle
$$\{z:|\arg (z-z_0)-\arg z_0|\le \alpha \},$$
where $\alpha \in (0,\pi/2)$ and $\epsilon$ do not depend on $z_0\in \cS$.
\end{lem}
%{\it Proof of Lemma~\ref{l1}.}
\begin{proof} 1$^o$.
%Denote $L_k=f^{-k}(\partial W_1)$, $k=0,1,...$.
%As $K_f$ is connected, all the critical points of the map
%$f: W_1\to W$ are contained in $K_f$. Hence,
%for each $k$, the set $L_k$ is the boundary of the simply-connected domain
%$f^{-k}(W_1)$, and
%it consists of the union $\Gamma_k$ of
%arcs of the equipotential of levels $b_i/D^{k}$, $1\le i\le s$,
%arcs of rays between these equipotentials and themselves or $J_P$,
%and subcontinua of $K_P$.
%If a ray $R$ has a limit point in $K_f$, then,
%for every $k$, it must cross the set $L_k=f^{-k}(\partial W_1)$
%at a single point denoted by $z_k(R)$, and this point lies in
%$\Gamma_k$.
%It follows, all limit points of $R$ belong to $K_f$.
Let $B_{*,k}=\sup\{u_P(z)|z\in L_k\}$.
For every
$k\ge 0$ there is a number $C_k$, such that, for every
ray $R\in\mathcal{R}$,
the length of the arc $R_k$ of $R$ between the points $z_k(R)$ and $z_{k+1}(R)$
is bounded by $C_k$. This is because the latter arc is an arc of a $P$-ray that joins two equipotentials of positive levels
$B_{*,k}$, $b_{*,k}$. Denote $\tilde L_k=\psi(L_k)$.
$\tilde L_k$ is a compact subset of $A_K$, which surrounds $\cS$.
By the above, every $K$-related arc $\tilde R$
splits into arcs $\tilde R_k=\psi(R_k)$, $k\ge 0$, i.e., $\tilde R_k$ is the arc
of $\tilde R$ joining
$\tilde z_k(\tilde R)$ and $\tilde z_{k+1}(\tilde R)$.
For every $k$,
the supremum of lengths over all arcs $\tilde R_k$
of the $K$-related rays $\tilde R$ is bounded
by the number
$\tilde C_k=C_k \sup\{|\psi'(z)|: z\in \overline W^*_1\setminus f^{-k-2}(W^*_1)\}$.

%From this point on, the proof is practically identical to
%the proof of Lemma 2.1 of~\cite{LP}. Here it is.

Let $A_{1,K}=\psi(W_1\setminus K_f)$
and
$
g = \psi\circ f\circ \psi^{-1} : A_{1,K}\to A_K
$
be a conjugated map. Then $z$ tends to $\cS$ if and only if $g(z)$ tends to
$\cS$.
%For a leaf $\gamma_0\in \Gamma$ that surrounds $K$,
%there is a component $\gamma_1$ of its $f$-preimage that also surrounds $K$
%and lies in the component $U(K)$ of $\bbc \setminus \gamma_0$ containing $K$.
%If we denote the component of $\bbc \setminus \gamma_1$ containing $K$ by
%$U(K)_1$ and $\gamma_0$ is chosen sufficiently close to $K$ that there are
%no critical points of $f$ in $U(K)\setminus K$  we have
%$$
%K = \bigcap_{n\ge 0}(f|_{U(K)_1})^{-n}(U(K)).
%$$
%The set $K$ is not a point by our assumption.
It is well-known e.g. \cite{P}, that then the map $g$ extends to an expanding holomorphic
map in an annulus
$
U_0=\{z: 1-\rho_0<|z|<1+\rho_0\},
$
for some $\rho_0>0$.
That means that after passing if necessary to an iterate of $g$ (which we also
denote $g$) we have
\begin{equation}\label{1}
|(g^{-1})^\prime(z)|< c<1
\end{equation}
for every $z\in U_0$ and for every branch $g^{-1}$ such that $g^{-1}(z)\in
U_0$.
% (we will consider only such branches).

Fix a set $\tilde L_m\subset U=A_K\cap U_0$,
for some $m$ large enough.
Then, for each $n=1,2,...$, $\tilde L_{n+m} = \{z\in U: g^n(z)\in
\tilde L_{n+m} \}$. Denote by $l_n$ the supremum
of lengths of $\tilde R_{n+m}$ over  all $R\in\mathcal{R}$. Note that each $l_n$ is finite,
because $l_n\le \tilde C_{m+n}$.
In fact, much more is true: as $g^n(\tilde R_{n+m})$ is $\tilde S_m$ for some ray $S\in\mathcal{R}$, ~(\ref{1}) gives us that $l_n< c^n\cdot
l_0$.
Given a $K$-related ray $\tilde{R}$, the length of its arc from the point
$\tilde z_m(\tilde R)$ to $\cS$, which is
in the component of ${\bf C}\setminus
\tilde\gamma_0$ containing $\cS$, is bounded from above by
$
\sum_{n=0}^\infty c^nl_0 <\infty$. Moreover, the same argument shows the following

{\bf Claim 1}. {\it Lengths of the arcs of K-related rays $\tilde R$ between $\tilde z_k(\tilde R)$ and $\cS$ tend uniformly to zero (exponentially in $k$)}.
%and lengths in the component bounded by $\tilde\gamma_t$ is bounded by
%$\sum_{n=t}^\infty a^n l_0$ which converges to 0 as $t\to\infty$.

2$^o$
Fix a closed non-degenerate arc $I\subset \cS$. There exists a $K$-related
ray
converging to a point of $I$. Otherwise no $K$-related ray ends in the
arc $g^n(I)$, for every $n$. This is impossible because $g^n(I)=\cS$ for
big $n$ and the set of $K$-related rays is non-empty
(for example, it contains images by $\psi$
of $P$-rays landing at repelling periodic points
of the polynomial-like map $f: W^*_1\to W^*$; for the existence of such $P$-rays, see \cite{Mil0}, \cite{EL89}, \cite{LP}).
We need to show that the set $K(I)$ of arguments
of all $K$-related rays ending in $I$ is closed.

This is an immediate consequence of the next claim that follows, basically, from Claim 1 and will be also useful later on.
Given a K-related ray $\tilde R_t$ of argument $t$ (i.e., $t\in\Lambda$)  consider its arc $\hat r_t$ which lies between
$\tilde L_0$ and $\cS$ and parameterized as a curve $\tilde r_t:[b_{*,0},0]\to A_K\cup\cS$ as follows.
%Here $b_j/D$ is the level of an equipotential where the arc $\psi^{-1}(\hat r_t)$ of a P-ray begins.
For any $b\in [b_{*,0},0)$, define the point $r_t(x)\in A_K$
as such that $\psi^{-1}(r_t(x))$ is a point of $P$-ray of argument $t$ and the equipotential level $b$.
Finally, let $\tilde r_t(0)=\lim_{b\to 0}\tilde r_t(x)\in\cS$ where the limit exists by the proven part 1$^0$.

{\bf Claim 2}. {\it The family $\tilde{\mathcal{R}}=\{\tilde{r}_t\}_{t\in\Lambda}$ is a compact subset of $C[b_{*,0},0]$}.

Let us first show that the family of functions $\{\hat{r}_t\}$ is equicontinuous. In view of Claim 1, this would follow from
the equicontinuity of the restricted family $\tilde{\mathcal{R}}_m=\{\hat{r}_t: [b_{*,0},b_{*,0}/D^{m}]\to A_K\}$, for each $m>1$ integer.
Fix $m$ and consider two objects: a compact set $E_m\subset \C$ bounded by the equipotential of levels $b_{*,0}$ and $b_{*,0}/D^m$ of $P$ and a family $\mathcal{R}_m$ of (closed) arcs in $E_m$ of all P-rays that join equipotential levels $b_{*,0}$, $b_{*,0}/D^m$ and
are parameterized by the equipotential level $b\in [b_{*,0}, b_{*,0}/D^m]$. It is very easy to see that this is a compact subset
of $C[b_{*,0}, b_{*,0}/D^m]$ (indeed, map this family by a fixed high iterate of $P$ to a family of smooth arcs of $P$-rays which are
preimages of segments of standard rays by the B\"{o}ttcher coordinate $B_P$ at infinity; hence, this new family is compact; then pull it back). As $\mathcal{R}_m\subset C[b_{*,0}, b_{*,0}/D^m]$ is compact, it is equicontinuous.
In turn, since $\psi^{-1}$ is a homeomorphism on $E_m$ (onto its image) and each $\psi^{-1}(\tilde {r}_t)\in\mathcal{R}_m$, the family $\tilde{\mathcal{R}}_m$ is equicontinuous too. Thus $\tilde{\mathcal{R}}$ is an equicontinuous family. It remains to prove that it is closed. So let a sequence $\hat{r}_{t_n}$ converge uniformly  in $[b_{*,0},0]$.
%to some function $\hat r$.
In particular, $\hat{r}_{t_n}$ crosses $\tilde L_k$ for each $k$ big enough.
One can assume that $t_n$ tends to some $t$. Then the sequence of arcs of P-rays $\psi^{-1}\circ\tilde{r}_{t_n}$,
on the one hand, tends, uniformly on each interval $[b_{*,0}, b_{*,0}/D^{m}]$, to an arc $r$ of P-ray of argument $t$, on the other hand, crosses each $L_k$ with $k$ big. Hence, $r$ has a limit point in $K_f$. Applying $\psi$ we get that
the limit of $\tilde{r}_{t_n}$ is a K-related arc which ends the proof of the claim.

%It will imply the statement since
%$\lambda(z_0)=\bigcap \lambda(I_1)$ over  all closed intervals
%$I_1$ covering $z_0$.
%It also impies the assertion on $\lambda(z_0)$ by the
%projection $\hat P:\hat\lambda(z_0)\to\lambda(z_0)$.
%Let $t_j\to t$, where $t_j$ is the argument of a $K$-related
%ray $\tilde R_j=\psi(R_{t_j})$,
%where, in turn, $R_{t_j}$ is a $P$-ray of argument $t_j$. Assume that $\{t_j\}\subset K(I)$, i.e., each $\tilde R_j$ lands in $I$.
%As $t_j\to t$, for every $k$, segments of the rays $R_{t_j}$
%between any two equipotentials of $P$ tend uniformly to a segment
%of the ray $R_t$ between the same equipotentials.
%It follows, that the ray $\tilde R=\psi(R_t)$ crosses every set $\Gamma_k$.
%Hence, $\tilde R$ has a limit point in $\cS$, that is, it is a $K$-related
%ray. By the above and by Claim 1 of 1$^o$, $\tilde R$ lands at $I$.

This proves 2$^0$ when $I$ is not a single point. By the intersection of compacta, 2$^o$ also holds
if $I$ is a point.

%The proof of 3$^o$ is identical to the one of \cite{LP}, Lemma 2.1.
 3$^o$ Every branch of $g^{-n}$ is a well defined univalent function in every disc contained in $U_0$. Hence, by Koebe distortion theorem e.g. \cite{Gol}, one can choose $0<\rho' <\rho_0$ such that for every
$$
z\in U'=\{z:1-\rho'<|z|<1+\rho'\},
$$
every $n=1,2,\cdots$ and every branch $g^{-n}$
\begin{equation}\label{2.2}
|\frac{(g^{-n})^\prime(x)}{(g^{-n})^\prime(y)}|<2
\end{equation}
whenever $|z-x|<\rho'$ and $|z-y|<\rho'$.

We reduce $U'$ further as follows. By Claim 1, fix $m_0>m$ such that the length of the arc of any K-related ray $\tilde R$ between $\tilde z_{m_0}(\tilde R)$ and $\cS$ is less than $\rho'$.
On the other hand, if $z$ lies in an unbounded component of $R\setminus z_{m_0}(R)$, i.e., in the arc of $R$ between $z_{m_0}(R)$ and $\infty$ then
$u_P(z)\ge b_{*, m_0}$, in particular, there is $r>0$ independent on $z$ and $R$ as above such that the distance between $z$ and $J_P$ is at least $r$.
Therefore, there exists some $\rho_1\in (0, \rho')$ such that for every $z\in \{z:1<|z|<1+\rho_1\}$, if $z$ belongs to a K-related ray $\tilde R$ then $z$ lies in an arc of $\tilde R$
between $\tilde z_{m_0}(\tilde R)$ and $\cS$.
%in other words, $z\in \tilde R_{k}$ for some $k\ge m_0$.
Let
$$
U_1=\{z:1-\rho_1<|z|<1+\rho_1\}.
$$

Introduce the following notations:

Given $x\in U_1$, denote by $l_x$ the part
of the $K$-related ray passing through $x$ between $x$ and $\cS$
(if such a ray exists). This notation is correct: as noted already before, if another $K$-related ray
passes through $x$ and next ramifies from $l_x$, it goes to a component
of $\psi(J(f))$, not to $\cS$. So it is not $K$-related.
%(the same argument was used already at the beginning of this Section).

Denote by $h_x$ the interval which joins $x$
and $\cS$,  orthogonal to $\cS$. By $l(x)$ and $h(x)$ denote the
corresponding
Euclidean lengths. Find a large enough $N$ so that  $\tilde\gamma_0:=\tilde L_N$ in $U_1$. By the choice of $U_1$,
\begin{equation}\label{2.3}
l(x)<\rho' \ \hbox{ for all} \ x \ \hbox{ between} \ \tilde\gamma_0 \ \hbox{
and} \ \cS.
\end{equation}
Let $\tilde\gamma_1=g^{-1}(\tilde\gamma_0)$.
There exists a positive $\beta_0$ less than $1$ such that
\begin{equation}\label{2.4}
\frac{h(x)}{l(x)}>\beta_0
\end{equation}
for all points $x$ in the annulus $V$ between $\tilde\gamma_0$ and $\tilde\gamma_1$.
%\in \tilde R_N$ over all $R\in\mathcal{R}$.
%Such $\beta_0>0$ exists as $l(x)<\rho$ and $\inf h(x)>0$ over all $x\in\tilde R_N$, $R\in\mathcal{R}$.
%in the strip $V$ between $\tilde\gamma_0$ and $\tilde\gamma_1$.

Fix the maximal $\epsilon _0 >0$ such that
$$U_2=\{z: 1-\epsilon _0 <|z|<1+\epsilon _0\}$$
does not
intersect $\tilde\gamma_1$. We intend to prove the assertion 3 of our Lemma
with
$$
\alpha = \arccos(\frac{\beta_0}{8L})
$$
where $L=\sup\{|g'(z)| : z\in U_0\}$
and with $\epsilon$ between $0$ and $\epsilon _0$ so small that
$1<|z|<1+\epsilon$ and $h(z)/|z-z_0|\ge 2\cos\alpha$ implies
$|\arg(z-z_0)-\arg z_0|\le \alpha$.

It is enough to prove that
\begin{equation}\label{enough}
\frac{h(x)}{l(x)}>\beta=\frac{\beta_0}{4L}
\end{equation}
for all $x\in U_2$. Assume the contrary: there exists $x_*\in U_2$,which belongs to some $K$-related ray $\tilde R$  with
\begin{equation}\label{2.5}
h(x_*)/l(x_*)\le \beta .
\end{equation}
Choose the minimal $n\ge 1$ such that $g^{n}(x_*)\in V$.
% for some $S\in\mathcal{R}$ (in fact, $S=f^n(R)$).
%(that is the $K$-related leaf $g^{kn}(\tilde\gamma_*)$ lies between $\tilde\gamma_0$ and $\tilde\gamma_1$).

The lengths $h^{(i)}$ and $l^{(i)}$ of the curves $g^{i}(h_{x_*})$ and
$g^{i}(l_{x_*})$ can not exceed $\rho'$ for all $i=0,1,...,n$.
This holds for $l^{(i)}$ by (\ref{2.3}), because $g^i(x_*)$ is
between $\tilde\gamma_0$ and $\cS$. We cope with $h^{(i)}$'s by induction:
Length$(h^{(0)})<\rho$ by the definition of $U_1$. If it holds for all $i\le
j-1$ then by (\ref{2.2})

$$
\frac{h^{(j-1)}}{l^{(j-1)}}\le 4\beta=\beta_0/L.$$
Then
$$ h^{(j)}\le L h^{(j-1)}\le \beta_0\cdot l^{(j-1)}<l^{(j-1)}<\rho'
$$
Now we use the assumption (\ref{2.5}) and again apply (\ref{2.2}), and we obtain for
$z_*=g^{n}(x_*)\in \tilde S_N$,
$$
\frac{h(z_*)}{l(z_*)}\le \frac{h^{(n)}}{l^{(n)}}\le 4\beta =
\beta_0/L<\beta_0.
$$
This contradicts (\ref{2.4}).
\end{proof}
%Lemma~\ref{l1} is proved.
\begin{com}
The key bound (\ref{enough}) can also be seen directly from (\ref{2.4}) (with, for instance, $\beta=\beta_0/10$) by applying, additionally to the Koebe distortion bound (\ref{2.2}),
another distortion bound as follows: there is a function $\epsilon:(0,1)\to (0,+\infty)$, $\epsilon(r)\to 0$ as $r\to 0$, such that for any univalent function $\varphi$ on the unit disc, if $\varphi(0)=0$, $\varphi'(0)=1$,  then $|\log\frac{\varphi(z)}{z}|<\epsilon(|z|)$, see e.g. \cite{Gol}. This bound is applied to a function
$\varphi(z)=\frac{g^{-n}(w+\rho_0 z)-g^{-n}(w)}{(g^{-n})'(w)\rho_0}$ where $n$ is the minimal so that $g^n(x)\in V$, $w\in\cS$ is the projection of $g^n(x)$ to $\cS$
and reducing $\rho'$. Note that
$(g^{-n})'(w)>0$ because $g$ preserves $\cS$.
\end{com}

\

We continue as follows (cf. the proof of Theorem 6.9 of~\cite{abcdef}).
%Let us extend the straightening map $h:V\to \C$ to a quasiconformal homeomorphism of $\C$ (denoted again by $h$) even
%though the conjugacy between $G$ and $f$ will only hold on the neighborhood $V_1$ of
%$K_f$. One can assume
Recall that the straightening $h:\C\to\C$ is a quasiconformal homeomorphism which is holomorphic at $\infty$ and $h'(\infty)\neq 0$. It conjugates the polynomial-like map $f$ with the polynomial $G$ near their filled Julia sets
$K_f$ and $K_G$ respectively. Let $B_G: A_G\to \D^*$ be the B\"{o}ttcher coordunate of $G$
such that $B_G(z)/z\to 1$ as $z\to\infty$, which is well defined in the basin of infinity $A_G=\C\setminus K_G$ of $G$ as $K_G$ is connected.

We have the following picture:
\begin{equation}\label{pict}
\D^*\xrightarrow{\psi^{-1}}\C\setminus K_f\xrightarrow{h}\C\setminus K_G\xrightarrow{B_G}\D^*.
\end{equation}
Consider a map $\Psi:=\psi\circ h^{-1}\circ B_G^{-1}: \D^*\to \D^*$ from
the uniformization plane of the polynomial $G$ to the $g$-plane of K-related rays. It is a
quasiconformal homeomorphism which is holomorphic at $\infty$.
%from a one-sided neighborhood
%of the unit circle $S^1$ in $\mathbb{D^*}$ into $\mathbb{D^*}$,
%which leaves $S^1$ invariant.
For $u\in \cS$, let $L_u=\Psi(r_u\cap
\mathbb{D^*})$ where $r_u=\{tu: t>0\}$ is a standard ray in the uniformization
plane of $G$.\footnote{Note that the curve $L_u$ lies in the left-hand disc $\D^*$ of (\ref{pict}) while the point $u$ is at the boundary of the right-hand disc there.}

%\begin{claim}\label{1.11} The curve $L$ tends non-tangentially
%to a unique point of the unit circle $S^1$.
%\end{claim}

%\smallskip

%\noindent \textbf{Claim A.} \emph{
\begin{lem}\label{l2} The curve $L_u$ converges non-tangentially to a
unique point $z_0=z_0(u)$ of the unit circle $\cS$. Moreover, there is $\beta \in (0,\pi/2)$
%and $\delta>0$
such that, for any $u\in \cS$ and all $z\in L_u$ close enough to $\cS$,
\begin{equation}\label{eqt2}
|\arg (z-z_0)-\arg z_0|\le \beta.
\end{equation}
%where $\alpha \in (0,\pi/2)$ and $\epsilon$ do not depend on $z_0\in S^1$.
Here $\beta$ depends only on the quasiconformal deformation of the straightening map $h$.
%{\it Proof of Lemma~\ref{l1}.}
Furthermore, for every $z_0\in \cS$
there exists a unique $u$ such that $L_u$ lands at $z_0$.
%}
\end{lem}
%\smallskip
\begin{proof} of Lemma \ref{l2}. (cf. \cite{abcdef}, Section 6).
%This follows from properties of quasiconformal mappings \cite{Ahl}.
$\Psi: \D^*\to \D^*$ extend to a homeomorphism of the closures $\bar\D^*$ onto $\bar\D^*$ and then to a quasiconformal
homeomorphism $\Psi^*$ of $\C$ by $\overline{\Psi^*(z)}=1/\Psi^*(1/\bar z)$, see \cite{Ahl}. Note that quasiconformal deformations of $\Psi$ and $\Psi^*$ are the same,
i.e. are equal to the quasiconformal deformation $M$ of the straightening map $h$. Consider the curve
$L^*_u=\Psi^*(r_u)$. It is an extension of the curve $L_u$, which crosses $\cS$
at a point $z_0=\Psi^*(u)$. As a quasiconformal image of a straight line,
the curve $L^*_u$ has the following property \cite{Ahl}: there exists $C=C(M)>0$,
such that $|z-z_0|/|z-1/\overline{z}|<C$, for every $z\in L^*_u$. Therefore, $L^*_u$
tends to $z_0$ non-tangentially, moreover, (\ref{eqt2}) holds for some $\beta=\beta(C(M))$. The last claim follows from the fact that
$\Psi^*$ is a homeomorphism.
\end{proof}
Now, define the correspondence $\lambda$ as follows (having in mind (\ref{pict})).
%\begin{equation}\label{pict}
%\D^*\xrightarrow{\psi^{-1}}\C\setminus K_f\xrightarrow{h}\C\setminus K_G\xrightarrow{B_G}\D^*.
%\end{equation}
Let $R$ be a $P$-ray to $K_f$. By Lemma~\ref{l1}, the
$K$-related ray $\tilde{R}=\psi(R)$ tends to a point $z_0\in \cS$.
%non-tangentially.
By Lemma \ref{l2}, there exists a unique $L_u$ which tends to $z_0$. The curve $\psi^{-1}(L_u)=h^{-1}\circ B_G^{-1}(\{tu: t>1\})$ is a polynomial-like ray $l_\tau$ where $u=e^{2\pi i \tau}$.
Let
$$\lambda(R):=  \psi^{-1}(L_u).$$
The correspondence $\lambda$ is "onto" by the first claim of Lemma \ref{l2} along with Lemma \ref{l1} (2$^o)$.

Now, both curves $\tilde{R}$, $L_u$ in $\D^*$ tend to the point $z_0\in\cS$ non-tangentially, by Lemma \ref{l1} and Lemma \ref{l2} respectively. Then, by definition, the P-ray $R$ and the polynomial-like ray $\lambda(R)$ converge to a single prime end of $K_{\bf f}$
non-tangentially, hence, $R$ and $\lambda(R)$ are also $K_{\bf f}$-equivalent.
%Hence, by Lindel\"{o}f's theorem, e.g., \cite{Pom},
%see e.g. \cite{Pom},
%$R=\psi^{-1}(\tilde{R})$ and $l=\psi^{-1}(L_u)$ have the same limit set in $K_f$.
%Since $\tilde{R}$ and $L_u$ are homotopic among the curves which tend to $z_0$
%within a certain Stolz angle $\{z\in\D^*: |\arg(z-z_0)-\arg z_0|\le \alpha\}$, for some $\alpha\in(0,\pi/2)$, the claim about homotopy follows. By the last claim of Lemma \ref{l2}, the map
%$\lambda$ is onto.
Finally, the condition that $R$ and $\lambda(R)$ are $K_{\bf f}$-equivalent
%$\pr(R)=\pr(\ell)$ and that
%$R$ and $l$ are homotopic outside $K_f$ among curves with the same limit set,
uniquely determines the polynomial-like ray $\lambda(R)$.

%\begin{com}\label{another}
%Assuming local rectifiability of polynomial-like rays $\{l_\tau\}_{\tau\in\cT}$,
%One can define the correspondence as in Theorem \ref{T1} more directly by considering a foliation of curves
%$\{\psi(l_{\tau})\}_{\tau\in\cT}$ in the $g$-plane of K-related rays and then
%modifying the proof of Lemma 2.1 in \cite{LP}
%so that all conclusion of the above mentioned Lemma 2.1 stay the same (except the part about the length).
%In this way the map $\Psi$ may not be necessary.
%The conclusion holds for the foliation $\{\psi(l_{\tau})\}$ to $\cS$.
%\end{com}
It remains to prove the ``almost injectivity'' of $\lambda$. This is a direct
consequence of the established above one-to-one correspondence between $K$-related rays and curves $L_u$ and the following claim whose proof is identical to the one of Theorem 6.8 of \cite{abcdef}
(for completeness, we reproduce it below with obvious changes in notations). While passing from K-related rays to $P$-rays we use the fact that if a K-related ray is periodic, the corresponding $P$-ray
converges to a periodic point of $P$ which is either repelling or parabolic (by Snail Lemma \cite{Mil0}, it cannot be irrationally indifferent).
\begin{lem}\label{Theorem 1'}
Any point $w\in \cS$ is the landing point of precisely one $K$-related ray,
except for when one and only one of the following holds:

\begin{enumerate}

\item[(i)] $w$ is the landing point of exactly two $K$-related rays,
which are non-smooth and have a common smooth arc that goes to the point $w$;

\item[(ii)] $w$ is a landing point of at least two disjoint $K$-related rays in which
    case $w$ is a (pre)periodic point of $g$ and some iterate $g^n(w)$
    belongs to a finite (and depending only on $K$) set $\hat Y$ of
    $g|_{\cS}$-periodic points each of which is the landing point of
    finitely many, but at least two, $K$-related rays, which are periodic of the same period depending merely on the landing point $w$.\footnote{In \cite{abcdef}, Theorem 6.8 (ii),
    it is claimed erroneously that all K-related rays to the point $w$ are smooth, cf. \cite{pz1}. Note that this claim is not relevant to the rest of \cite{abcdef}}
%$E$-related ray that lands at a point of the set $\hat Y(E)$ is smooth.
%in other words, $\Pi(R)$ is a point $a$, and, for some $n\ge 0$,
%$P^{pn}(a)$ hits a finite set of periodic orbits of $P^p: K\to K$,
%each of which is either repelling or parabolic.
\end{enumerate}

Moreover, if $w$ is periodic then {\rm (i)} cannot hold.
\end{lem}
\begin{proof}
%A consideration we are going to use in the proof
%is typical in the lamination theory originated in [Thu85], see
%e.g. the proof of Lemma 2.4 of [BL02a].
Assume that there are two $K$-related rays landing at a point $w\in \cS$ and
that (i) does not hold. We need to prove that then (ii) holds. If (i) does not
hold, then there exist disjoint $K$-related rays landing at $w$. Let us study
this case in detail.

Associate to any such pair of rays $\hat R_t$, $\hat R_{t'}$ an open arc $(\hat R_t, \hat R_{t'})$ of $\cS$
%$\mathbb{S}^1$ (where $\mathbb{S}^1$ is the circle at infinity in the $g$-plane)
as
follows. Two points of $\mathbb{S}^1$ with the arguments $t, t'$ split
$\cS$ into two arcs. Let the arc $(\hat R_t, \hat R_{t'})$ be the one
of them that contains no arguments of $K$-related rays except for possibly
those that land at $w$. Geometrically, it means the following. The $K$-related
rays $\hat R_t, \hat R_{t'}$ together with $w\in \cS$ split the plane into two
domains. The arc $(\hat R_t, \hat R_{t'})$ corresponds to the one of them,
disjoint from $\cS$. Let $L(\hat R_t, \hat R_{t'})=\delta$ be the angular length of
$(\hat R_t, \hat R_{t'})$. Clearly, %$0<L(\hat R_t, \hat R_{t'})<1$.
$0<\delta<1$. Now we make a few observations.

(1) \emph{If $K$-related disjoint rays of arguments $t_1, t_1'$ land at a
common point $w_1$ while $K$-related disjoint rays of arguments $t_2, t_2'$
land at a point $w_2\not=w_1$, then the arcs $(\hat R_{t_1}, \hat R_{t_1'})$,
$(\hat R_{t_2}, \hat R_{t_2'})$ are disjoint.}

This follows from the definition of the arc
$(\hat R_{t}, \hat R_{t})$.

(2) \emph{If disjoint $K$-related rays $\hat R_t, \hat R_{t'}$ of arguments $t,
t'$ land at a common point $w$, then $K$-related rays $g(\hat R_t), g(\hat
R_{t'})$ are also disjoint and land at the common point $g(w)$. Moreover,}
%$$L(g(\hat R_t), g(\hat R_{t'}))\ge \min\{d^p L(\hat R_t, \hat
%R_{t'})(\text{mod}\, 1),
%1-d^p L(\hat R_t, \hat R_{t'})(\text{mod}\, 1)\}>0.$$
$$L(g(\hat R_t), g(\hat R_{t'}))\ge \min\{D\delta(mod 1),
1-D\delta(mod 1)\}>0.$$

Indeed, the images $g(\hat R_t), g(\hat R_{t'})$ are disjoint near $g(w)$,
because $g$ is locally one-to-one. Hence, $g(\hat R_t)\cap g(\hat
R_{t'})=\emptyset$.
because otherwise
the corresponding $P$-rays would have their limit sets
in different components of $K_P$, a contradiction since
both rays $g(\hat R_t)$, $g(\hat R_{t'})$ are $K$-related.
%The inequality of (2) is a consequence of the definition of
%$L(R_t, R_{t'})$ and the fact that $g$ is conjugate to $P^p$
%in a one-sided neighborhood of $S^1$.
Since the argument of $g(\hat R_t)$ is represented by the point $Dt(mod 1)\in (0,1)$, we get
the inequality of (2).

Let us consider the following set $\hat Z(K)$ of points in $\cS$: $w\in \hat
Z(K)$ if and only if there is a pair of disjoint $K$-related rays $\hat R, \hat
R'$, which both land at $w$, and such that $L(\hat R, \hat R')\ge 1/(2D)$.
Denote by $\hat Y(K)$ a set of periodic points which are in forward images of
the points of $\hat Z(K)$.

(3) \emph{If the set $\hat Z(K)$ is non-empty, then it is finite,
and consists of (pre)periodic points.}

Indeed, $\hat Z(K)$ is finite by (1). Assume $w\in \hat Z(K)$. Then, by (2)
some iterate $g^n(w)$ must hit $\hat Z(K)$ again.

To complete the proof, choose disjoint $K$-related rays $\hat R_t, \hat R_{t'}$
landing at $w\in \cS$ and use this to prove that all claims of (ii) hold.

{\it We show that the orbit $w, g(w),\dots$ cannot be infinite.} Indeed,
otherwise by (1)-(2), we have a sequence of non-degenerate pairwise disjoint
arcs $(g^n(\hat R_t), g^n(\hat R_{t'}))\subset \cS$, $n=0,1,...$. By (2), some
iterates of $w$ must hit the finite set $\hat Z(K)$ and hence $\hat Y(K)$
(which are therefore non-empty), a contradiction.

{\it Hence for some $0\le n<l$, $g^n(w)=g^l(w)$; let us verify that other
claims of {\rm(ii)} holds.} Replacing $w$ by $g^n(w)$, we may assume that $w$ is a
(repelling) periodic point of $g$ of period $k=l-n$. By (2), $w\in \hat Y(K)$.
By Theorem 1,~\cite{LP}, the set of $K$-related rays landing at $w$ is
finite, and each $K$-related ray landing at $w$ is periodic with the same
period.
%By Lemma~\ref{cap}, each such ray is also smooth. %Then we replace the
%$E$-related rays $\hat R_t, \hat R_{t'}$ by two disjoint adjacent $E$-related
%rays landing at $w$.
%Then the sequence
%$(g^n(\hat R_t), g^n(\hat R_{t'}))$, $n=0,1,...$, is a sequence
%of pairwise disjoint non-degenerate arcs.
%By (2)-(3), an iterate of $w$ must hit a point $a$ of the finite set $\hat
%Y(E)$.
Hence, (ii) holds. Finally, the last claim of the lemma follows because a periodic non-smooth ray must have infinitely many broken points, hence, no other ray can have a joint arc with it that goes up to the Julia set, see Lemma 6.1, \cite{abcdef} for details.
\end{proof}

\section{Proof of Theorems \ref{t1}-\ref{t2}}\label{s3}
\subsection{Theorem \ref{t1}.}
Part (a) is an immediate corollary of Lemma~\ref{l1}) and Lindel\"of's theorem, as in \cite{LP}.
%If a point $a\in J_{\bf f}$ is accessible along
%a curve $s$ in $\C\setminus K_{\bf f}$, then $a$ is the landing point
%of a $P$-ray which is homotopic to $s$ among curves in $\C\setminus K_{\bf f}$ converging to $a$.
Indeed, since a
curve $s\subset W\setminus K_{\bf f}$ converges to a point $a\in K_{\bf f}$,
the curve
$\tilde s=\psi(s)$ converges to a point $z_0\in\cS$, and the limit
of the function $\psi^{-1}$ along the curve $\tilde s$ exists and equals $a$.
By  Lemma \ref{l1}, there is a $K$-related ray $\tilde R$ that tends to $z_0$, moreover, non-tangentially. Then, by Corollary 2.17 of \cite{Pom}, the $P$-ray $R$ converges to the same point $a$.
By definition, curves $s$, $R$ are $K_{\bf f}$-equivalent.

Let us prove part (b). A closed set $S\cup K_f$ is connected and its complement is connected, too (by the Maximum Principle). Consider the set $\hat{S}=\psi(S)\subset \D^*$. Let $I=\overline{\hat{S}}\setminus\hat{S}$. Then $I$ is a connected closed subset of the unit circle $\cS$. Let us prove $I$ is a single point. Otherwise there is an interior point $x\in I$ which is $g$-periodic. Let $\beta$ be a $K$-related ray that lands at $x$.
Notice that since $x$ is an interior point of $I$, then $\beta$ must cross $\hat{S}$.
Now, since $x$ is $g$-periodic, $R=\psi^{-1}(\beta)$ is a periodic $P$-ray, hence, it converges to a periodic point $a\in\overline{S}\setminus S$ of $P$ and crosses $S$, a contradiction since $S\subset K_P$. It proves that $I$ is a single point. Denote it by $z_0$.
%The curve $s\subset W\setminus K_f$ converges to the point $a\in J_f$,
%and $s\subset K_P$.
%Hence, the curve
%$\hat s=\Phi(s)$ converges to a point $z_0\in S^1$, and the limit
%of the function $\psi^{-1}$ along the curve $\tilde s$ exists and equals $a$.
%Moreover, every $K$-related ray landing at $z_0$ is disjoint with $\hat s$,
%because rays and $K_P$ are disjoint.
Choose two sequences $z'_n, z''_n$ of $\cS$ tending to $z_0$
from the left and from the right respectively, and two sequences of
$K$-related rays $l'_n, l''_n$, so that $l'_n$ lands $z'_n$ and $l''_n$
lands at $z''_n$. Then, passing perhaps to subsequences, by Claim 2, see the proof of Lemma \ref{l1}, the sequence $l'_n$ tends to a $K$-related ray
$l'$ and $l_n''$ tends to $K$-related ray
$l''$, where $l'$ and $l''$ land at the same $z_0$.
By the above, $l', l''$ are disjoint. Now we apply Lemma~\ref{Theorem 1'}
to conclude that the point $z_0$ is $g$-(pre-)periodic,
and some iterate of $z_0$ lies in a finite set $\hat Y\subset \cS$
of periodic points,
which is independent on
$z_0$. Hence, the
point $a$ is $P$-(pre-)periodic, and some iterate of $a$
lies in a finite set $Y\subset J_f$
of periodic points,
which is independent on
$a$. As every point of $Y$
is a landing point of a periodic ray,
it can be either repelling or parabolic.
%(Snail lemma is used here, see \cite{Mil0}).
%This proof covers also the claim in Comment~\ref{b1}.
%Indeed, as the rays $R',R''$ land at $a\in J_f$, the $K$-related
%rays $l'=\Phi(R'), l''=\Phi(R'')$ converges to a point $z_0\in S^1$
%(here, we don't use the property (pl2)). Then Theorems 6.8,~\cite{abcdef} applies.
%By  Lemma 1.3$^o$ we can apply Lindel\"of's theorem
%(on the existence of the nontangential limit) to every
%$K$-related ray ending at the point $z_0$. By
%Lemma 1.2$^o$ at least one such ray exists what completes the proof.
\subsection{Theorem \ref{t2}}\label{s4}
Proof of (b), (c):
%$\Lambda$ is closed by 2$^o$ of Lemma \ref{l1}.
It follows from the definition of $\Lambda$
that $\sigma_D(\Lambda)=\Lambda$ and $\sigma_m\circ p=p\circ\sigma_{D}$ on $\Lambda$. By the invariance and since $\Lambda\neq \cT$, the set $\Lambda$ contains no intervals;
(c) is a reformulation of a part of the statement of Theorem \ref{T1}.

Proof of (a), (d): considering $\Lambda$ as a subset of ${\cS}=\{|z|=1\}$ define a new map $p_K: \Lambda\to \cS$ as follows: for $\tau\in\Lambda$, let $p_K(\tau)\in \cS$ be the landing point of a K-related ray of argument $\tau$. Let us recall the map $\Psi=\psi\circ h^{-1}\circ B_G^{-1}:\D^*\to\D^*$ which was introduced in the proof of Theorem \ref{T1}, and its quasi-conformal extension $\Psi^*:\C\to\C$. By Lemma \ref{l2} and the definition of maps $\lambda$ and $p$, we have:
\begin{equation}\label{*}
p_K=\Psi^*|_{\cS}\circ p.
\end{equation}
Since $\Psi^*: \cS\to \cS$ is an orientation preserving homeomorphism, it is enough to prove (a), (d) if one replaces the map $p$ by $p_K$. By 2$^o$, Lemma \ref{t1}, $p_K^{-1}(I)$ is closed in $\cS$ for any closed arc $I\subset \cS$. Therefore, $\Lambda=p_K^{-1}(\cS)$ is closed and
the map $p_K:\Lambda\to \cS$ is continuous. To show (d),
define an extension $\tilde p_K:\cS\to \cS$ of $p_K:\Lambda\to \cS$ in an obvious way as follows. Let $J:=(t_1,t_2)$ be a component of ${\cS}\setminus \Lambda$. Then $p_K(t_1)=p_K(t_2):=w_J$ because otherwise there would be a point of $\cS$ with no K-related rays landing at it.
Let $\tilde p_K(\tau)=w_J$ for all $\tau\in J$. Then $\tilde p_K:\cS\to \cS$ is continuous. Now, given $t\in\cS$, the set $\tilde p_K^{-1}(\{t\})$ is either a singleton or a non-trivial closed arc. This follows from the definition of $\tilde p_K$ and because K-related rays with different arguments don't intersect unless case (i) of Theorem \ref{Theorem 1'} takes place. Therefore,
$\tilde p_K:\cS\to\cS$ is monotone and degree one.

Proof of (e): let $\tilde h$ be another straightening, $\tilde{\Psi}:\D^*\to\D^*$ be the corresponding to $\tilde h$ quasiconformal map and $\tilde{\Psi}^*:\C\to\C$ its
quasiconformal extension.
As $p_K:\cS\to\cS$ is independent on the straightening, by (\ref{*}), $\tilde p=T|_{\cS}\circ p$ where $T=(\tilde{\Psi}^*)^{-1}\circ\Psi^*$. On the other hand,
being considered on $\D^*$, $T=(B_G\circ\tilde h)\circ(B_G\circ h)^{-1}$, hence, $T$ commutes with $z\mapsto z^m$ for $|z|>1$ near $\cS$, by definitions of $h,B_G$. Therefore, a homeomorphism $\nu:=T|_{\cS}:\cS\to\cS$ commutes
with $z\mapsto z^m$ on $\cS$, too. It is then well known that $\nu(z)=v z$, for some $v\in\C$ with modulus $1$ such that $v^m=v$ (proof: as $\nu(1)^m=\nu(1)$ let $v=\nu(1)$, so that a homeomorphism $\nu_0=v^{-1}\nu:\cS\to\cS$
commutes with $z\mapsto z^m$ too and $\nu_0(1)=1$; then there is a lift $\tilde\nu_0:\R\to\R$ of $\nu_0$ such that $\tilde\nu_0(0)=0$, $\tilde\nu_0-1$ is $1$-periodic and $\tilde\nu_0(mx)=m\tilde\nu_0(x)$ for all $x\in\R$ which, in turn, implies
$\tilde\nu_0(n/m^k)=n/m^k$ for all $n,k\in\Z$ positive; by continuity, $\tilde\nu_0(x)=x$ for all $x$).

\end{document}